\documentclass[12pt]{article}
\usepackage{amsmath,amsxtra,amssymb,latexsym, amscd,amsthm}
\usepackage[top=3.5cm, bottom=3.0cm, left=3.5cm, right=2cm] {geometry}

\advance\voffset-1truecm\relax
\advance\hoffset-1truecm\relax

\def\N{{\Bbb N}}

\def\P{{\Bbb P}}

\def\Q{{\Bbb Q}}

\numberwithin{equation}{section}

\newcommand {\Cal}{\mathcal}

\newtheorem{lemma}{Lemma}[section]
\newtheorem{theorem}[lemma]{Theorem}

\newtheorem{definition}[lemma]{Definition}
\newtheorem{remark}[lemma]{Remark}
\def\leq{\leqslant}

\begin{document}
\title{Schmidt's subspace theorem  for moving hypersurface targets in subgeneral position with index in algebraic variety }

\author{Tingbin Cao\thanks{The first author is supported by the National Natural Science Foundation of China (\#11871260).} and Nguyen Van Thin\thanks{Corresponding author: Nguyen Van Thin.}}
\date{}
\maketitle

\vspace{-0.5cm}
\begin{abstract}
 Recently, Xie-Cao \cite{XC} obtained a Second Main Theorem for moving hypersurfaces located in subgeneral position with index which is extended the result of Ru \cite{R3}. By using some methods due to Son-Tan-Thin \cite{STT}, Quang \cite{Q2} and Xie-Cao \cite{XC}, we shall give a Schmidt's Subspace Theorem for moving hypersurface targets in subgeneral position with index intersecting algebraic variety. Our result is a extension the Schmidt's Subspace Theorem due to Son-Tan-Thin \cite{STT} and Quang \cite{Q2}.

\noindent {\it Keywords}: Diophantine approximation, Schmidt's Subspace Theorem.\\
 Mathematics Subject Classification 2010. 11J68, 11J25, 11J97.

\end{abstract}

\section{Some definitions and result}
Let $k$ be an algebraic number field of degree $d$. Denote $M(k)$ by the set of places (i.e., equivalent classes of absolute values) of $k$ and write
$M_{\infty}(k)$ for the set of Archimedean places. From $v \in M(k)$, we choose the normalized absolute value $| . |_{v}$ such that
$| . |_{v}=| . |$ on $\Q$ (the standard absolute value) if $v$ is archimedean, whereas for $v$ non-archimedean $|p|_{v}=p^{-1}$
if $v$ lies above the rational prime $p$. Denote by $k_v$ the completion of $k$ with respect to $v$ and by $d_v=[k_v: \Q_v]$ the local
degree. We put $\Vert.\Vert_{v}=| . |_{v}^{d_v/d}$. Then norm $||.||_v$ satisfies the following properties:

$(i)$ $||x||_v\ge 0$, with equality if and only if $x=0;$

$(ii)$ $||xy||_v=||x||_v||y||_v$ for all $x, y \in k;$

$(iii)$ $||x_1+\dots+x_n||_v\le B_v^{n_v}\cdot \max \{||x_1||_v, \dots, ||x_n||_v\}$ for all $x_1, \dots, x_n \in k$, $n\in \N$, where $n_v=d_v/d$, $B_v=1$ if $v$ is non-archimedean and $B_v=n$ if $v$ is archimedean.\\
Moreover, for each $x\in k\setminus \{0\}$, we have the  following product formula:
$$ \prod_{v\in M(k)}\Vert x\Vert_{v}=1. $$
For $v\in M(k)$, we also extend $\Vert .\Vert_{v}$ to an absolute value on the algebraic closure $\overline{k}_v.$\par

For $x\in k$, the  logarithmic height of $x$ is defined by $h(x)=\sum_{v\in M(k)} \log^{+}\Vert x\Vert_{v} $, where $\log^{+}\Vert x\Vert_{v}= \log \max \{\Vert x\Vert_{v}, 1\}.$ For ${\bf x}=[x_0: \dots : x_M] \in \mathbb P^M(k)$, we denote $x=(x_0,\dots,x_M)\in k^{M+1}.$ We set $\Vert x\Vert_{v}=\max_{0 \le i \le M}\Vert x_i\Vert_{v},$ and define the logarithmic height of ${\bf x}$ by
\begin{eqnarray}\label{ct11s}
h({\bf x})=\sum_{v\in M(k)} \log \Vert x\Vert_{v}.
\end{eqnarray} For a positive integer $d$, we set
\begin{eqnarray*}
\Cal T_d := \big\{ (i_0,\dots,i_M) \in \N_0^{M+1}:
i_0 + \dots + i_M = d \big\}.
\end{eqnarray*}
Let $Q$ be a  homogeneous polynomial of degree $d$ in $k[x_0,\dots, x_M].$ We write
\begin{eqnarray*}
Q= \sum\limits_{I \in \Cal T_{d}} a_{I}x^I.
\end{eqnarray*}
A (fixed) hypersurface $D$ of degree $d$ in $\mathbb P^M(k)$ is the zero set of a homogeneous polynomial $Q,$ that is
$$ D=\{(x_0:\dots:x_M)\in \mathbb P^M(k): Q(x_0,\dots,x_n)=0\}.$$
 Set $\Vert Q\Vert_v:=\sum_{I\in \mathcal T_d}\Vert a_I\Vert_v.$ The height of $Q$ (or $D$) is denoted by $h(Q)$ (or $h(D)$) and given by
\begin{eqnarray*}
h(Q):=\sum_{v\in M(k)}\log\Vert Q\Vert_v.
\end{eqnarray*}
Let $D$ be a hypersurface which is defined by the zero of homogeneous polynomial $Q.$  For each $v\in M(k),$ the Weil function $\lambda_{D,v}$ is defined by
\begin{eqnarray*}
\lambda_{D,v}({\bf x}):=\log\frac{\Vert x\Vert^d_v\cdot \Vert Q\Vert_v}{\Vert Q(x)\Vert_v}, \quad {\bf x}\in \P^M(k)\setminus D.
\end{eqnarray*}

Let $\Lambda$ be an infinite index set.  We call  a moving hypersurface $D$ in $\P^M(k)$ of degree $d,$ indexed by $\Lambda$ is collection of hypersurfaces $\{D(\alpha)\}_{\alpha\in\Lambda}$ which are defined by the zero of  homogeneous polynomials $\{Q(\alpha)\}_{\alpha\in \Lambda}$ in $k[x_0,\dots,x_M]$ respectively. Then, we can write $Q=\sum_{I\in\mathcal T_d}a_I x^I,$ where $a_I$'s are functions from $\Lambda$ into $k$ having no common zeros points, and called it by moving homogeneous polynomial.\par

Throughout  this paper, we consider an infinite index $\Lambda;$ a set $\Cal Q:=\{D_1,\dots,D_q\}$ of moving hypersurfaces in $\P^M(k),$ indexed by $\Lambda,$ which are defined by the zero of moving homogeneous polynomials $\{Q_1,\dots,Q_1\}$ in $k[x_0,\dots,x_M]$ respectively, an arbitrary projective variety $V\subset\P^M(k)$ of dimension $n$ generated by the homogeneous  ideal $\Cal I(V)$. We write  $$Q_j= \sum_{I\in \mathcal T_{d_j}}a_{j, I}x^{I}\;(j=1,\dots,q),\;\text{where}  \;d_j=\deg Q_j.$$\par

Let $A\subset\Lambda$ be an infinite subset and denote by $(A,a)$ each set-theoretic map $a: A\to k.$
Denote by $\mathcal R_{A}^{0}$  the set of equivalence classes of pairs $(C,a)$, where $C\subset A$ is a subset with finite complement and $a: C \to k$ is a map; and the equivalence relation is defined as follows: $(C_1, a_1) \sim (C_2, a_2)$
if there exists $C \subset C_1\cap C_2$ such that $C$ has finite complement in $A$ and $a_1|_{C}=a_2|_{C}.$ Then $\mathcal R_{A}^{0}$ has an obvious ring structure. Moreover, we can embed $k$ into $\mathcal R_{A}^{0}$ as constant functions.\par

\begin{definition} \label{defi1.2s} For each $j\in\{1,\dots,q\}$, we write $\mathcal T_{d_j}=\{I_{j,1}, \dots, I_{j, M_{d_j}}\},$ where $M_{d_j}:=\binom{d_j+M}{M}.$ A subset $A\subset \Lambda$ is said to be coherent with respect to
$\Cal Q$ if for every polynomial
$ P\in k[x_{1,1}, \dots, x_{1, M_{d_1}}, \dots, x_{q,1}, \dots, x_{q, M_{d_q}}] $
that is homogeneous in $x_{j,1}, \dots, x_{j, M_{d_j}}$ for each $j=1, \dots, q$, either\\
$ P(a_{1, I_{1,1}}(\alpha), \dots, a_{1, I_{1, M_{d_1}}}(\alpha), \dots, a_{q, I_{q,1}}(\alpha), \dots, a_{q, I_{q, M_{d_q}}}(\alpha))$
vanishes for all $\alpha \in A$ or it vanishes for only finitely many $\alpha \in A.$
\end{definition}

By \cite[Lemma 2.1]{RV}, there exists an infinite coherent subset $A\subset \Lambda$ with respect to $\Cal Q.$ For each $j\in\{1,\dots,q\},$ we fix an index $I_{j}\in\Cal T_{d_j}$ such that $a_{j, I_{j}}\not\equiv  0$ (this means that
$a_{j, I_{j}}(\alpha)\ne 0$ for all, but finitely many, $\alpha\in A$), then $\dfrac{a_{j, I}}{a_{j, I_{j}}}$ defines an element of $\mathcal R_{A}^{0}$ for any
 $I \in \mathcal T_{d_j}$. This element given by the following function:
$$ \{\alpha \in A: a_{j, I_{j}}(\alpha) \ne 0\} \to k, \quad\alpha \mapsto  \frac{a_{j, I}(\alpha)}{a_{j, I_{j}}(\alpha)}.$$
Moreover, by coherent, the subring of $\mathcal R_{A}^{0}$ generated over $k$ by such all elements is an integral domain \cite[page 3]{G}. We define $\mathcal R_{A, \mathcal Q}$ to be the field of fractions of that integral domain.\par

Denote by $\Cal A$ the set of all functions $ \{\alpha \in A: a_{j, I_{j}}(\alpha) \ne 0\} \to k, \quad\alpha \mapsto  \frac{a_{j, I}(\alpha)}{a_{j, I_{j}}(\alpha)}$ and $k_{\Cal Q}$ the set of all formal finite sum
$\sum_{m=1}^st_m\prod_{i=1}^sc_i^{n_i},$ where $t_m\in k,$ $c_i\in \Cal A, n_i\in\N.$\par

Each pair $(\widehat b,\widehat c)\in k_{\Cal Q}^2,$  ($\hat c(\alpha)\ne 0$ for all,  but finitely many,  $\alpha\in A$) defines a set-theoretic function, denoted by $\frac{\hat b}{\hat c}$, from $\{\alpha: \widehat c(\alpha)\ne 0\}$ to $k$,   $\frac{\widehat b}{\widehat c}(\alpha):=\frac{\widehat b(\alpha)}{\widehat c(\alpha)}.$  Denote by $\widehat{\mathcal R}_{A, \mathcal Q}$ the set of all such functions. Each element $a\in\mathcal R_{A, \mathcal Q}$ is a class of some functions $\widehat{a}$ in $\widehat{\mathcal R}_{A, \mathcal Q}.$ We call that $\widehat{a}$ is a special representative of $a.$  It is clear that for any two special representatives $\widehat a_1,\widehat a_2$ of the same element $a\in\mathcal R_{A, \mathcal Q},$ we have $\widehat a_1(\alpha)=\widehat a_2(\alpha)$ for all, but finitely many $\alpha\in A.$ For a polynomial $P:=\sum_I a_Ix^I\in\mathcal R_{A,\Cal Q}[x_0, \dots, x_M],$ assume that $\widehat a_I$ is a special representative  of $a_I.$ Then $\widehat P:=\sum_I \hat a_Ix^I$ is called  a special representative of $P.$  For each $\alpha\in A$ such that all functions $\widehat a_I$'s are well defined at $\alpha,$ we set $\widehat P(\alpha):=\sum_I \hat a_I(\alpha)x^I\in k[x_0,\dots,x_M];$ and we also say that the special representative $\widehat P$  is well defined at $\alpha.$ Note that  each special representative $\widehat P$ of $P$ is well defined at all, but finitely many, $\alpha\in A.$  If  $\widehat P_1, \widehat P_2$ are two special presentatives of $P,$ then   $\widehat P_1(\alpha)= \widehat P_2(\alpha)$ for all, but finitely many $A.$\par

\begin{definition}\label{defi1.5s}
A sequence of points $x=[x_0:\dots : x_M]: \Lambda \to V$ is said to be $V-$algebraically non-degenerate with respect to $\mathcal Q$ if  for each infinite coherent subset $A\subset \Lambda$ with respect to $\mathcal Q$, there is no  homogeneous polynomial $P\in \mathcal R_{A,\Cal Q}[x_0, \dots, x_M] \setminus \mathcal I_{A, \Cal Q}(V)$  such that
$\widehat P(\alpha)(x_0(\alpha), \dots, x_M(\alpha)) =0$ for all, but finitely many, $\alpha\in A$ for some (then  for all) representative $\widehat P$ of $P,$ where $\mathcal I_{A, \Cal Q}(V)$ is the ideal in $\mathcal R_{A, \mathcal Q}[x_0,\dots,x_M]$ genarated by $\Cal I(V).$
\end{definition}

\begin{definition}\label{new1}
Let $m\ge n$ be a positive integer. We say that the family moving hypersurfaces $\Cal Q$  is in $m$-strictly subgeneral position with respective to $V$ if for any subset $I\subset \{1,\dots,q\}$ with $\# I\le m+1,$
$$ \dim \Big(\cap_{i\in I}D_i(\alpha)\cap V(\overline k)\Big)\le m-\#I$$
for all, but finitely many $\alpha\in\Lambda,$ where $\overline k$ is the algebraic closure of $k.$
\end{definition}

From Definition \ref{new1}, we get

\begin{remark}\label{rm1}
Let $m\ge n$ be a positive integer. The family moving hypersurfaces $\Cal Q$  is in $m$-strictly subgeneral position with respective to $V$ implies that  for each $1\le j_0<\dots <j_m \le q, q\ge m+1$,  the system of equations
$$ Q_{j_i} (\alpha)(x_0, \dots, x_M)=0, \hspace{1cm} 0\le i \le m$$
has no solution $(x_0, \dots, x_M)$ satisfying $(x_0:\cdots: x_M) \in V(\overline k),$ for all, but finitely many $\alpha\in\Lambda.$  It means that $\Cal Q$  is in $m$-subgeneral position with respective to $V.$
\end{remark}

\begin{definition}\label{new2}
Let $\kappa$ be a positive integer such that $\kappa\le \dim V.$ We say that the family moving hypersurfaces $\Cal Q$  is in $m$-strictly subgeneral position with index $\kappa$ in $V$ if $\Cal Q$  is in $m$-strictly subgeneral position with respective to $V$  and for any subset $I\subset \{1,\dots,q\}$ such that $\# I\le \kappa,$
$$ \text{codim} \Big(\cap_{i\in I}D_i(\alpha)\cap V(\overline k)\Big)\ge \# I$$
for all, but finitely many $\alpha\in\Lambda.$ Here we set $\dim \emptyset=-\infty.$
\end{definition}

In 1997, Ru-Vojta \cite{RV} established the following Schmidt subspace theorem for the case of moving hyperplanes in projective spaces.\par\vskip 6pt

\noindent {\bf Theorem A.}  {\it Let $k$ be a number field and let $S\subset M(k)$ be a finite set containing all archimedean places. Let $\Lambda$ be an infinite index set and let $\mathcal H:=\{H_1, \dots, H_q\}$  be a set of moving  hyperplanes  in $\mathbb P^{M}(k),$ indexed by $\Lambda.$ Let ${\bf x}=[x_0: \dots: x_M]: \Lambda \to \mathbb P^{M}(k)$ be a sequence of points. Assume that

$(i)$ ${\bf x}$ is linearly nondegenerate with respect to $\mathcal H,$ which mean, for each infinite coherent subset $A\subset \Lambda$ with respect to $\Cal H,$ $x_0|_A,\dots,x_M|_A$ are linearly independent over $\Cal R_{A,\Cal H},$\par

$(ii)$ $h(H_j(\alpha))=o(h({\bf x}(\alpha)))$ for all $\alpha\in \Lambda$ and $j=1, \dots, q$ (i.e. for all $\delta>0$, $h(H_j(\alpha))\leq\delta h({\bf x}(\alpha))$ for all, but finitely many, $\alpha\in \Lambda).$\par

Then, for any $\varepsilon >0,$ there exists an infinite index subset $A\subset \Lambda$ such that
$$ \sum_{v\in S}\max_K\sum_{J\in K}\lambda_{H_j(\alpha), v}({\bf x}(\alpha)) \le (M+1+\varepsilon)h({\bf x}(\alpha))$$
holds for all $\alpha \in A.$ Here the maximum is taken over all subsets $K$ of $\{1,\dots,q\}$, $\#K=M+1$ such that $H_j(\alpha), j\in K$ are linearly independent over $k$ for each $\alpha\in\Lambda.$} \par\vskip 6pt

One of the most important developments in recent years in Diophantine approximation is Schmidt's subspace theorems for fixed hypersurfaces of Corvaja-Zannier \cite{CZ} and Evertse-Ferretti \cite{EF, EF2}. Motivated by their paper, Ru \cite{R2,R3} obtained important results  on the Second Main Theorem for fixed hypersurfaces. Later, Dethloff-Tan \cite{DT} generalized these Second Main Theorems of  Ru to the case of moving results. Basing on method of Dethloff-Tan \cite{DT},  Chen-Ru-Yan \cite{CRY3} and Le \cite{G}  extended Theorem A to the case of moving hypersurfaces in the projective spaces. In 2019, Dethloff-Tan \cite{DT1} extended the result of Ru \cite{R3} for moving target. Via Vojta's dictionary, the Second Main Theorem in Nevanlinna theory corresponds to Schmidt's Subspace Theorem in Diophantine approximation. In 2018, Son-Tan-Thin \cite{STT} gave the counterpart of Second Main Theorem due to Dethloff-Tan \cite{DT1} in Diophantine approximation. In 2018, Quang \cite{Q2} obtainted the Schmidt's subspace theorem for moving hypersurface in subgeneral position. In 2019, Xie-Cao \cite{XC} have been obtained the Second Main Theorem for moving hypersurfaces located in subgeneral position with index.\par

In this paper, by using some methods due to Son-Tan-Thin \cite{STT} and Quang \cite{Q2}, Xie-Cao \cite{XC}, we give a Schmidt's Subspace Theorem for moving hypersurfaces in $m$-subgeneral position with index $\kappa$ as follows:

\begin{theorem}\label{Schmidt} Let $k$ be a number field and let $S \subset M(k)$ be a finite set containing all archimedean places. Let ${\bf x}=[x_0: \dots: x_M]: \Lambda \to V$ be a sequence of points. Assume that\par

$(i)$ $\Cal Q$ is in $m$-strictly subgeneral position with index $\kappa$ in $V$  and ${\bf x}$ is $V-$ algebraically nondegenerate with respect to $\mathcal Q;$\\par

$(ii)$ $h(D_j(\alpha))=o(h({\bf x}(\alpha)))$ for all $\alpha\in\Lambda$ and $j=1, \dots,q$ (i.e. for all $\delta>0$, $h(D_j(\alpha))\leq\delta h({\bf x}(\alpha))$ for all, but finitely many, $\alpha\in \Lambda).$\par

Then, for any $\varepsilon >0,$ there exists an infinite index subset $A\subset \Lambda$ such that
$$ \sum_{v\in S}\sum_{j=1}^{q}d_j^{-1}\lambda_{D_j(\alpha), v}({\bf x}(\alpha)) \le \left(\left(\dfrac{m-n}{\max \{1,\min \{m-n,\kappa\}\}}+1\right)(n+1)+\varepsilon\right)h({\bf x}(\alpha))$$
holds for all $\alpha \in A.$
\end{theorem}

\begin{remark}\label{remark}  (i) By replacing $Q_j$ by $Q_j^{\frac{d}{d_j}},$ where $d=\text{lcm}\{d_1,\dots,d_q\},$ in Theorem \ref{Schmidt}, we may assume that $D_1,\dots,D_q$ have the same degree $d.$

(ii) By replacing $Q_j= \sum_{I\in \mathcal T_{d}}a_{j, I}x^{I}$ by $Q'_j = \sum\limits_{I \in \Cal T_{d}} \dfrac{a_{jI}}{a_{jI_{j}}}x^I$, in Theorem \ref{Schmidt}, we may assume that $Q_j\in\mathcal R_{A, \mathcal Q}[x_0,\dots,x_M].$
\end{remark}
Our result is a extension the Schmidt's Subspace Theorem due to Son-Tan-Thin \cite{STT} and Quang \cite{Q2}.

\section{Some Lemmas}
We write
\begin{eqnarray*}
Q_j = \sum\limits_{I \in \Cal T_{d}} a_{jI}x^I, \quad (j = 1,\dots,q).
\end{eqnarray*}
Let $A\subset\Lambda$ be an infinity coherent subset with respect to $\Cal Q.$  For each $j\in\{1,\dots,q\},$ we fix an index $I_{j}\in\Cal T_{d}$ such that $a_{j, I_{j}}\not\equiv  0$ (this means that
$a_{j, I_{j}}(\alpha)\ne 0$ for all but finitely many $\alpha\in A$), then $\dfrac{a_{jI}}{a_{jI_{j}}}$ defines an element of $\mathcal R_{A}^{0}$ for any
 $I \in \mathcal T_{d}$.
 Set \begin{eqnarray*}
Q'_j = \sum\limits_{I \in \Cal T_{d}} \dfrac{a_{jI}}{a_{jI_{j}}}x^I, \quad (j = 1,\dots,q).
\end{eqnarray*}
\noindent Let $ t=(\dots, t_{jI},\dots)$ be a family of variables.
 Set
\begin{eqnarray*}
\widetilde{Q_j} = \sum\limits_{I \in \Cal T_{d}} t_{jI}x^I\in k[t,x].
\end{eqnarray*}
We have
$\widetilde{Q_j}(\dots,\dfrac{a_{jI}}{a_{jI_{j}}}(\alpha),\dots,x_0,\dots, x_M)=Q'_j(\alpha)(x_0,\dots,x_M)$ for all $\alpha\in A$ outside a finite subset.

Assume that the ideal $\Cal I(V)$ of $V$ is generated by homogeneous polynomials $\mathcal P_1,\dots,\mathcal P_s.$ Since $\Cal Q$ is in $m$-subgeneral position on $V$, for each $J:=\{j_0,\dots,j_m\}\subset\{1,\dots,q\},$ there is a subset  $A_J\subset A$ with finite complement such that for all $\alpha\in A_J,$ the homogeneous polynomials
 $$ \mathcal P_1,\dots,\mathcal P_s,Q'_{j_0}(\alpha),\dots,Q'_{j_m}(\alpha)\in k[x_0,\dots,x_M]$$
have no common non-trivial solutions in $\overline{k}^{M+1}.$
Denote by $$_{k[t]}(\mathcal P_1,\dots, \mathcal P_s, \widetilde{Q}_{j_0},\dots,\widetilde{Q}_{j_m})$$ the ideal in the ring of  polynomials in $x_0,\dots, x_M$ with coefficients in $k[t]$ generated by
$\mathcal P_1,\dots,\mathcal P_s,\widetilde{Q}_{j_0},\dots,\widetilde{Q}_{j_m}.$ A polynomial $\widetilde R$ in $k[t]$  is called an
 {\it inertia form} of the polynomials $$ \mathcal P_1,\dots, \mathcal P_s, \widetilde{Q}_{j_0},\dots,\widetilde{Q}_{j_m}$$ if it has the following property:
\begin{eqnarray*}
x_i^N\cdot \widetilde {R}\in \;_{k[t]}(\mathcal P_1,\dots, \mathcal P_s, \widetilde {Q}_{j_0},\dots,\widetilde{Q}_{j_m})
\end{eqnarray*} for $i=0,\dots,M$ and for some non-negative integer $N$  (see e.g.  \cite{Z}). It follows from the definition that the set $\Cal I$ of inertia forms of polynomials $\mathcal P_1,\dots,$ $\mathcal P_s,$ $\widetilde{Q}_{j_0},$ $\dots,$ $\widetilde{Q}_{j_m}$ is an ideal in $k[t].$\par

It is well known that $(s+m+1)$ homogeneous polynomials
$$ \mathcal P_i(x_0,\dots,x_M), \widetilde{Q_j}(\dots,t_{jI},\dots,x_0,\dots, x_M), i\in\{1,\dots,s\},$$ $j\in J$
have no common non-trivial solutions in $x_0,\dots,x_M$
for special values $t_{jI}^0$ of $t_{jI}$ if and only if there exists an inertia form $\widetilde {R_J}^{t_{jI}^0}$  such that
$\widetilde R_J^{t_{jI}^0}(\dots,t_{jI}^0,\dots)\ne 0$ (see e.g. \cite[page 254]{Z}). For  each $\alpha\in A_J,$ choose  $\widetilde {R_J}^{\alpha}\in\Cal I$ with respect to the special values $t_{jI}^\alpha:= \dfrac{a_{jI}}{a_{jI_{j}}}(\alpha).$ Set $R_J^{\alpha}:=\widetilde R_J^{\alpha}(\dots,\dfrac{a_{jI}}{a_{jI_{j}}},\dots).$ Then $R_J^{\alpha}$  is a special presentative of an element  $\mathcal R_{A,\Cal Q}.$
By construction, we have
\begin{eqnarray}\label{t1}
R_J^{\alpha}(\alpha)\ne 0 \;\text{ for all} \;\alpha\in A_J \;(\text{and hence, for all, but finitely many,}\;\alpha\in A).
\end{eqnarray}
Since $k[t]$ is Noetherian,  $\Cal I$ is generated by finite polynomials $\widetilde{R}_{J_1},\dots,\widetilde{R}_{J_p}.$ For each $\alpha,$ we write
$\widetilde {R_J}^{\alpha}=\sum_{\ell=1}^p \widetilde{G}^{\alpha}_\ell\widetilde {R}_{J_\ell},$  $\widetilde G^{\alpha}_{\ell}\in k[t].$ We have that $G^\alpha_{\ell}:=\widetilde G^{\alpha}_\ell(\dots,\dfrac{a_{jI}}{a_{jI_{j}}},\dots),$ and $R_{J\ell}:=\widetilde {R}_{J\ell}(\dots,\dfrac{a_{jI}}{a_{jI_{j}}},\dots)$ are special representatives of elements in $\mathcal R_{A,\Cal Q}.$ It is clear that $R_J^{\alpha}=\sum_{\ell=1}^p G^{\alpha}_{\ell} R_{J\ell}.$
Hence, by (\ref{t1}), we have
\begin{eqnarray*}
0\ne R_J^{\alpha}(\alpha)=\sum_{\ell=1}^pG^\alpha_{\ell} (\alpha) R_{J_\ell}(\alpha)
\end{eqnarray*}
for all, but finitely many $\alpha\in A.$
Therefore, there is $\ell_0\in\{1,\dots,p\}$ such that
\begin{eqnarray}\label{t2}
R_{J_{\ell_0}}(\alpha)\ne 0 \quad \text{for all, but finitely many}\quad \alpha\in A.
\end{eqnarray}
Furthermore, by the definition of the inertia forms, there are a non-negative integer $N,$ polynomials
 $b_{i\ell}\in\mathcal R_{A, \Cal Q}[x_0, \dots, x_M]$ with $\deg b_{ij_k}=N-d$ and $\deg b_{i\ell}=N-\deg P_\ell$ such that
\begin{eqnarray}\label{ct17s}
R_{J_{\ell_0}}\cdot x_i^{N}=\sum_{k=0}^{m}b_{i j_k}\cdot Q'_{j_k}+\sum_{\ell=1}^{s}b_{i\ell}\cdot \mathcal P_\ell,
\end{eqnarray}
for all $0\le i\le M$.

Let ${\bf x}: \Lambda \to V\subset P^{M}(k)$ be a map. A map $(C, a) \in \mathcal R_{A}^{0}$ is called small with respect to ${\bf x}$ if and only if
$$ h(a(\alpha))=o(h({\bf x}(\alpha))), $$
which means that, for every $\varepsilon >0$, there exists a subset $C_{\varepsilon}\subset C$ with finite complement such that $h(a(\alpha))\le \varepsilon h({\bf x}(\alpha))$ for all $\alpha \in C_{\varepsilon}.$ We denote by $\mathcal K_{\bf x}$ the set of all such small maps. Then, $\mathcal K_{\bf x}$ is subring of $ \mathcal R_{A}^{0}$. It is not an entire ring, however, if $(C, a)\in \mathcal K_x$ and $a(\alpha)\ne 0$ for all but finitely $\alpha \in C$, then we have $(C\setminus \{\alpha: a(\alpha)=0\}, \dfrac{1}{a})\in \mathcal K_{\bf x}.$
Denote by $\mathcal C_{\bf x}$ the set of all positive functions $g$ defined over $\Lambda$ outside a finite subset of $\Lambda$ such that
$$ \log^{+}(g(\alpha)) =o(h({\bf x}(\alpha))).$$
Then $\mathcal C_{\bf x}$ is a ring. Moreover, if $(C, a)\in \mathcal K_{\bf x}$, then for every $v\in M(k)$, the function $||a||_v: C\to \mathbb R^{+}$ given by $\alpha \mapsto ||a(\alpha)||_v$ belongs to $\mathcal C_{\bf x}$. Furthermore, if $(C, a)\in \mathcal K_{\bf x}$ and $a(\alpha)\ne 0$ for all but finitely $\alpha \in C$, the function $g: \{\alpha|a(\alpha) \ne 0\}\mapsto \dfrac{1}{||a(\alpha)||_v}$ also belongs to $\mathcal C_{\bf x}.$

From (\ref{t2}) and (\ref{ct17s}), similarly to Lemma 2.2 in \cite{G}, under the assumption of Theorem \ref{Schmidt} and Remark \ref{remark}, we have the following result.
\begin{lemma}\label{lem21s} Let $A\subset \Lambda$ be coherent with respect to $\Cal Q.$ Suppose that
 $\mathcal Q$ is in $m-$subgeneral position with respective to $V.$ Then for each $J\subset\{1,\dots,q\}$ with $|J|=m+1,$ there are functions $l_{1, v}, l_{2, v}\in\mathcal C_{\bf x}$ such that
$$ l_{2, v}(\alpha)||x(\alpha)||_v^{d}\le \max_{j\in J}||Q_j(\alpha)(x(\alpha))||_v\le l_{1, v}(\alpha)||x(\alpha)||_v^{d}$$
for all $\alpha \in A$ ouside finite subset and all $v\in S.$
\end{lemma}

For each positive integer $\ell$ and for each vector sub-space $W$ in $k[x_0, \dots, x_M]$ (or in $\mathcal R_{A, \Cal Q}[x_0, \dots, x_M]$), we denote by $W_\ell$ the vector space consisting of all homogeneous polynomials in $W$ of degree $\ell$ (and  of the zero polynomial).

By the usual theory of Hilbert polynomials, for $N>>0$, we have
\begin{eqnarray*} H_V(N)&:=&\dim_k \dfrac{k[x_0, \dots, x_M]_N}{\mathcal I(V)_N}\\
&=&\dim_{\overline {k}} \dfrac{\overline {k}[x_0, \dots, x_M]_N}{\mathcal I(V(\overline k))_N}
=\deg V. \dfrac{N^n}{n!}+O(N^{n-1}).
\end{eqnarray*}

Using the method in the proof of \cite[Lemma 3.1 ]{Q2} and \cite[Lemma 3.3]{XC}, we get the result as follow:
\begin{lemma}\label{lm1}
Let $Q_1, \dots, Q_{m+1}$ be homogeneous polynomials in $\mathcal R_{\Lambda}^{0}[x_0, \dots, x_M]$ of the same degree $d\ge 1,$
 in $m-$subgeneral position with index $\kappa$  in $V.$ Then for an infinitely subset $A\subset \Lambda$ which is coherent with respect to $\{Q_1,\dots, Q_{m+1}\},$
 there exists $n$ homogeneous polynomials $P_1=Q_1,\dots,P_{\kappa}=Q_{\kappa},P_{\kappa+1},\dots, P_{n+1}$ in $\mathcal R_{\Lambda}^{0}[x_0, \dots, x_M]$ of the form
$$ P_t=\sum_{j=\kappa+1}^{m-n+t}c_{tj}Q_j, c_{tj}\in k,\quad t=\kappa+1, \dots, n+1, $$ such that the system of equations
$$ P_{t} (\alpha)(x_0, \dots, x_M)=0, \hspace{1cm} 1\le t \le n+1$$
has no solution $(x_0, \dots, x_M)$ satisfying $(x_0:\cdots: x_M) \in V(\overline k),$ for all  $\alpha\in\Lambda$ outside a finite subset of $\Lambda.$
\end{lemma}

\begin{proof}
We assume that $Q_i\; (1\le i\le m+1)\;$ has the following form:
$$ Q_i=\sum_{I\in\mathcal T_d}a_{i,I}x^{I}.$$
By the definition of $m$-strictly subgeneral position, there exists an infinite subset $A'$ of $A$ with finite complement such that, for all $\alpha\in A',$ the system of equations
$$ Q_i(\alpha)(x_0,\dots,x_M)=0,\;\,\, 1\le 1\le m+1,$$
has only the trivial solution $(0,\dots,0)$ in $\overline k.$ We note that $A'$ is coherent with respect to $\{Q_i\}_{i=1}^{m+1}.$ Replacing $A$ by $A'$ if necessary, we may assume that $A'=A.$ We fix $\alpha_0\in A.$ For each homogeneous polynomial $Q\in k[x_0,\dots, x_M],$ we will denote
$$D^{*}=\{(x_0:\dots:x_M)\in\mathbb P^M(\overline k): Q(x_0,\dots,x_M)=0\}.$$ Setting $P_1=Q_1,\dots,P_{\kappa}=Q_{\kappa},$ we build $P_{\kappa+1},\dots,Q_{n+1}$ as follows. First we see that
$$ \dim \left(\cap_{i=1}^{t}D_i^{*}(\alpha_0)\cap V(\overline k)\right)\le m-t\; \text{for}\; t=m-n+\kappa+1,\dots,m+1,$$
where $\dim \emptyset=-\infty$ and for any $1\le l\le \kappa,$ we have
\begin{eqnarray*}
\dim\Big(\cap_{j=1}^{l}D_j^{*}(\alpha_0)\cap V(\overline k)\Big)=n-l.
\end{eqnarray*}
Indeed, it is obvious $\dim \Big(\cap_{j=1}^{\kappa}D_j^{*}(\alpha_0)\cap V(\overline k)\Big)\ge n-l,$ and by the definition of $m$-strictly subgeneral position with index $\kappa,$ we get $\dim \Big(\cap_{j=1}^{l}D_j^{*}(\alpha_0)\cap V(\overline k)\Big)\le n-l.$ In particular, we have
\begin{eqnarray*}
\dim\Big(\cap_{j=1}^{\kappa}D_j^{*}(\alpha_0)\cap V(\overline k)\Big)=n-\kappa.
\end{eqnarray*}

{\bf Step 1.} We will construct $P_{\kappa+1}$ as follows. For each irreducible component $\Gamma$ of dimension $n-\kappa$ of $\cap_{j=1}^{\kappa}D_j^{*}(\alpha_0)\cap V(\overline k),$ we put
\begin{eqnarray*}
V_{1\Gamma}=\Big\{c=(c_{\kappa+1},\dots,c_{m-n+\kappa+1})\in k^{m-n+1},\Gamma\subset D_c^{*}(\alpha_0),\; \text{where}\; Q_c=\sum_{j=\kappa+1}^{m-n+\kappa+1}c_jQ_j\Big\}.
\end{eqnarray*}
Then $V_{1\Gamma}$ is a subspace of $k^{m-n+1}.$ Since
$$ \dim \Big(\cap_{i=1}^{m-n+\kappa+1}D_i^{*}(\alpha_0)\cap V(\overline k)\Big)\le n-\kappa-1,$$
there exists $i\in \{\kappa+1,\dots,m-n+\kappa+1\}$ such that $\Gamma\not\subset D_i^{*}(\alpha_0).$ Since the set of irreducible components of dimension $n-\kappa$ of $\cap_{j=1}^{\kappa}D_j^{*}(\alpha_0)\cap V(\overline k)$ is at most countable, we have
$$ k^{m-n+\kappa}\setminus \cup_{\Gamma}V_{1\Gamma}\ne \emptyset.$$ Hence there exists $(c_{1(\kappa+1)},\dots,c_{1(m-n+\kappa+1)})\in k^{m-n+1}$ such that $\Gamma\not\subset D_{\kappa+1}^{*}(\alpha_0)$ for all irreducible components of dimension $n-\kappa$ of $\cap_{j=1}^{\kappa}D_j^{*}(\alpha_0)\cap V(\overline k),$ where $$P_{\kappa+1}=\sum_{j=\kappa+1}^{m-n+\kappa+1}c_{1(j)}Q_j.$$ This clearly implies that
\begin{eqnarray}\label{m1}
 \dim \Big(\cap_{i=1}^{\kappa+1}D_i^{*}(\alpha)\cap V(\overline k)\Big)\le n-\kappa-1.
\end{eqnarray}
Certainly, we have
\begin{eqnarray}\label{m2}
\dim \Big(\cap_{i=1}^{\kappa+1}D_i^{*}(\alpha)\cap V(\overline k)\Big)\ge n-\kappa-1.
\end{eqnarray}
From (\ref{m1}) and (\ref{m2}), we get
\begin{eqnarray}
\dim \Big(\cap_{i=1}^{\kappa+1}D_i^{*}(\alpha)\cap V(\overline k)\Big)=n-\kappa-1.
\end{eqnarray}

{\bf Step 2.} We will construct $P_{\kappa+2}$ as follows. For each irreducible component $\Gamma'$ of dimension $n-\kappa-1$ of $\cap_{j=1}^{\kappa+1}D_j^{*}(\alpha_0)\cap V(\overline k),$ we put
\begin{eqnarray*}
V_{1\Gamma'}=\Big\{c=(c_{\kappa+1},\dots,c_{m-n+\kappa+2})\in k^{m-n+2},\Gamma'\subset D_c^{*}(\alpha_0),\; \text{where}\; Q_c=\sum_{j=\kappa+1}^{m-n+\kappa+2}c_jQ_j\Big\}.
\end{eqnarray*}
Then $V_{1\Gamma'}$ is a subspace of $k^{m-n+2}.$ Since
$$ \dim \Big(\cap_{i=1}^{m-n+\kappa+2}D_i^{*}(\alpha_0)\cap V(\overline k)\Big)\le n-\kappa-2,$$
there exists $i\in \{\kappa+1,\dots,m-n+\kappa+2\}$ such that $\Gamma\not\subset D_i^{*}(\alpha_0).$ Since the set of irreducible components of dimension $n-\kappa-1$ of $\cap_{j=1}^{\kappa+1}D_j^{*}(\alpha_0)\cap V(\overline k)$ is at most countable, we have
$$ k^{m-n+\kappa+1}\setminus \cup_{\Gamma'}V_{1\Gamma'}\ne \emptyset.$$ Hence there exists $(c_{2(\kappa+1)},\dots,c_{2(m-n+\kappa+2)})\in k^{m-n+2}$ such that $\Gamma'\not\subset D_{\kappa+2}^{*}(\alpha_0)$ for all irreducible components of dimension $n-\kappa-1$ of $\cap_{j=1}^{\kappa}D_j^{*}(\alpha_0)\cap V(\overline k),$ where $$P_{\kappa+2}=\sum_{j=\kappa+1}^{m-n+\kappa+2}c_{2(j)}Q_j.$$ This clearly implies that
$$ \dim \Big(\cap_{i=1}^{\kappa+2}D_i^{*}(\alpha_0)\cap V(\overline k)\Big)\le n-\kappa-2$$
and we also have $$ \dim \Big(\cap_{i=1}^{\kappa+2}D_i^{*}(\alpha_0)\cap V(\overline k)\Big)= n-\kappa-2.$$\par

Repeat again the above steps, after $(n+1-\kappa)$-th step, we get the hypersurface $P_{\kappa+1},\dots, P_{n+1}$ satisfying
$$ \dim \Big(\cap_{i=1}^{t}D_i^{*}(\alpha_0)\cap V(\overline k)\Big)\le n-t$$
for all $t=\kappa+1,\dots,n+1.$ In particular,
\begin{eqnarray}\label{ct1}
\cap_{i=1}^{n+1}D_i^{*}(\alpha_0)\cap V(\overline k)=\emptyset.
\end{eqnarray}
We denote $R$ by the inertia form of $\mathcal P_1,\dots,\mathcal P_s,\widetilde{P}_{1},\dots,\widetilde{P}_{n+1}.$
As in (\ref{t2}) and (\ref{ct1}), we see that $R(\alpha_0)\ne 0.$  By the property coherent of $A,$ it implies that $R(\alpha)\ne 0$ for all $\alpha\in A$ outside a finite set. By property of inertia form, we have
 \begin{eqnarray}\label{ct2}
\cap_{i=1}^{n+1}D_i^{*}(\alpha)\cap V(\overline k)=\emptyset
\end{eqnarray}
for all $\alpha\in A$ outside a finite set.
\end{proof}


Note that $c_{tj}$ does not depend on $\Lambda$ for all $ t=2,\dots, n+1$ and $ j=2, \dots, m+1,$  then $P_t[x_0, \dots, x_M]\in \mathcal R_{A, \Cal Q }[x_0, \dots, x_M] $ for $t=1, \dots, n+1.$ We denote $\mathcal I_{A, \Cal Q}(V)$  the ideal  in $\mathcal R_{A, \Cal Q }[x_0, \dots, x_M]$ generated by the elements in $\mathcal I(V).$  It is clear that $\mathcal I_{A, \Cal Q}(V)$ is also the sub-vector space of $\mathcal R_{A, \Cal Q}[x_0, \dots, x_M]$ generated by $\mathcal I (V ).$\par

We use the lexicographic order in $\N_0^n$ and for $I=(i_1,\dots,i_n),$ set $\Vert I\Vert :=i_1+\cdots+i_n.$
\begin{definition}\label{dn1}
For each
 $I=(i_1,\cdots, i_n)\in \N_0^n$ and $N\in\N_0$ with $N\geq d\Vert I\Vert,$ denote by $\Cal L_N^I$
the set of all $\gamma\in\mathcal R_{A, \Cal Q }[x_0,\dots,x_M]_{N-d\Vert I\Vert}$ such that
\begin{eqnarray*}
P_1^{i_1}\cdots P_n^{i_n}\gamma-
\sum_{E=(e_1,\dots,e_n)>I}P_1^{e_1}\cdots P_n^{e_n}\gamma_E\in \mathcal I_{A, \Cal Q }(V)_N.
\end{eqnarray*}
for some $\gamma_E\in \mathcal R_{A, \Cal Q} [x_0,\dots,x_M]_{N-d\Vert E\Vert}$.
\end{definition}

Denote by $\Cal L^I$ the homogeneous ideal in $\mathcal R_{A, \Cal Q } [x_0,\dots,x_M]$ generated by $\cup_{N\geq d\Vert I\Vert}\Cal L_N^I.$\par

\begin{remark}\label{r1}
\noindent i) $\Cal L_N^I$ is a $\mathcal R_{A, \Cal Q }$-vector sub-space of $\mathcal R_{A, \Cal Q }[x_0,\dots,x_M]_{N-d\Vert I\Vert},$ and
 $$(\Cal I(V), P_1,\dots,P_n)_{N-d\Vert I\Vert}\subset \Cal L_N^I,$$   where $(\Cal I(V), P_1,\dots,P_n)$ is
 the ideal in $\mathcal R_{A, \Cal Q }[x_0,\dots,x_M]$ generated by
 $\Cal I(V)\cup\{P_1,\dots,P_n\}.$\par

ii) For any $\gamma\in \Cal L_N^I$ and $P\in \mathcal R_{A, \Cal Q }[x_0,\dots,x_M]_k,$ we have $\gamma\cdot P\in\Cal L_{N+k}^I$\par

iii) $\Cal L^I\cap \mathcal R_{A, \Cal Q }[x_0,\dots,x_M]_{N-d\Vert I\Vert}=\Cal L_N^I.$\par

iv) $\frac{\mathcal R_{A, \Cal Q }[x_0, \dots, x_M]}{\Cal L^I}$
is a graded module over the graded ring $\mathcal R_{A, \Cal Q }[x_0,\dots,x_M].$

v) If $I_1-I_2=(i_{1,1}-i_{2,1}, \dots, i_{1,n}-i_{2, n})\in N_0^{n},$ then $\Cal{L}_{N}^{I_{2}}\subset \Cal{L}_{N+d||I_{1}||-d||I||_2}^{I_{1}}.$ Hence
$\Cal L^{I_2}\subset \Cal L^{I_1}.$
\end{remark}

Using Definition \ref{dn1}, by arguments as \cite{STT}, we get the following result.
\begin{lemma}\label{L4s}\cite{STT}
For all $N>>0$ divisible $d,$ there are homogeneous polynomials $\phi_1,$ $\dots,$ $\phi_{H_V(N)}$  in $ \mathcal R_{A, \Cal Q}[x_0,\dots,x_M]_N$ such that they form a basis of the $\mathcal R_{A, \Cal Q}$-vector space
$\frac{\mathcal R_{A, \Cal Q}[x_0,\dots,x_M]_N}{\mathcal I_{A, \Cal Q}(V)_N},$  and
 \begin{eqnarray*}
\prod_{j=1}^{H_V(N)}\phi_j-\big(P_{1}\cdots P_{n}\big)^{\frac{\deg V\cdot N^{n+1}}{d\cdot (n+1)!}-u(N)}\cdot P\in\mathcal I_{A, \Cal Q}(V)_N,
\end{eqnarray*}
where  $u(N)$ is a function  satisfying $u(N)\leq O(N^n)$,  $P \in \mathcal R_{A, \Cal Q}[x_0,\dots,x_M]$ is a homogeneous polynomials of degree
$$N\cdot H_V(N)-\frac{n\cdot\deg V\cdot N^{n+1}}{(n+1)!}+u(N)=\frac{\deg V\cdot N^{n+1}}{(n+1)!}+O(N^n).$$
\end{lemma}

\section{Proof of Theorem \ref{Schmidt}}
\begin{proof}
By \cite[Lemma 2.1 ]{CRY3}, there exists an infinite index subset $A\subset \Lambda$ which is coherent with respect to $\mathcal Q.$   By Remark \ref{remark}, we may assume that the polynomials $Q_j$'s have the same degree $d\ge 1$ and their coefficients  belong to the field  $\mathcal R_{A, \Cal Q}.$  By the fact that for any infinite subset $B$ of $A$, then $B$ is also coherent with respect to $\mathcal Q$ and $\Cal R_{B,\Cal Q}\subset\Cal R_{A,\Cal Q},$ in our proof, we may freely pass to infinite subsets. For simplicity, we still denote these infinite subsets by $A.$
We set
$$ \Cal J=\{(i_1, \dots, i_{m+1}); 1\le i_j\le q\; \text{for}\; j=1, \dots, q, i_1, \dots, i_{m+1}\; \text{distinct}\}. $$
For each $J=(i_1,\dots, i_{m+1})\in \Cal J,$ we denote by $P_{J, 1}\dots, P_{J,n+1}$ the moving homogeneous polynomials obtained in Lemma \ref{lm1} with respective to the family of moving hypersurfaces $\{Q_{i_1}, \dots, Q_{i_{m+1}}\}.$ It is easy to see that, for each $v\in S,$ there exists a positive function $h_v\in \mathcal C_{\bf x}$ such that
\begin{eqnarray}\label{ctm1}
||P_{J, t}(x(\alpha))||_v\le h_v(\alpha)\max_{1\le j\le m-n+t}||Q_{i_j}(x(\alpha))||_v
\end{eqnarray}
for all $\alpha\in A$ outside a finite subset, for all $t=\kappa+1,\dots, n+1.$ Here, the function $h_v$  may choose common for all $J\in \mathcal J.$
From the assumption, for each $a\in \mathcal R_{A, \Cal Q},$ and $v\in M(k)$, we have, for all $\alpha \in A$,
\begin{eqnarray}\label{ct21s}
\log ||a(\alpha)||_v\le \sum_{v\in M(k)} \log^{+}||a(\alpha)||_v=h(a(\alpha))\le o(h(x(\alpha))).
\end{eqnarray}
For each $v\in S$, and $\alpha \in A$, there exist a subset $J(v, \alpha)=\{i_1(v, \alpha), \dots, i_{q}(v, \alpha)\}$ of $\{1,\dots,q\}$ such that
\begin{eqnarray*}
  ||Q_{i_1(v, \alpha)}(\alpha)(x(\alpha))||_v\le   ||Q_{i_2(v, \alpha)}(\alpha)(x(\alpha))||_v \dots& \le ||Q_{i_{q}(v, \alpha)}(\alpha)(x(\alpha))||_v.
\end{eqnarray*}
We have $J(v, \alpha)=(i_1(v, \alpha), \dots, i_{m+1}(v, \alpha))\subset \mathcal J.$  Denote by $P_{J(v,\alpha), 1}, \dots, P_{J(v,alpha),n+1}$ by the moving homogeneous polynomials obtained in Lemma \ref{lm1} and by $D_{J(v,\alpha), j}$ the hypersurface defined by $P_{J(v,\alpha), j},$ $j=1,\dots,n+1.$  Note that, we may freely pass to infinite subsets, then we may assume
$D_{J(v, \alpha),1}, \dots, D_{J(v, \alpha),n+1}$ are in general position with respective $V.$
Apply to Lemma \ref{lem21s} for $\mathcal Q=\{D_{J(v, \alpha),1}, \dots, D_{J(v, \alpha),n+1}\},$ and using (\ref{ctm1}), there exist functions $g_{0, v}^{J(v, \alpha)}, g_v^{J(v, \alpha)} \in \mathcal C_{\bf x}$ (depend on $l_{2, v}$)  such that
\begin{eqnarray}\label{ctm2}
||x(\alpha)||_{v}^{d}\le g_{0, v}^{J(v, \alpha)}(\alpha)\max_{1\le j\le n+1}||P_{J(v, \alpha), j}(\alpha)(x(\alpha))||_v\le g_v^{J(v, \alpha)}(\alpha) ||Q_{i_{m+1}}(\alpha)(x(\alpha))||_v.
\end{eqnarray}
Therefore, from (\ref{ctm2}), we get
\begin{eqnarray}\label{ctm3}
&&\prod_{i=1}^{q}\dfrac{||x(\alpha)||_{v}^{d}}{||Q_{i}(\alpha)(x(\alpha))||_v}\\\nonumber
&\le& (g_v^{J(v, \alpha)})^{q-m}(\alpha)\prod_{j=1}^{m}\dfrac{||x(\alpha)||_{v}^{d}}{||Q_{i_j(v, \alpha)}(\alpha)(x(\alpha))||_v}
\notag\\\nonumber
&=&(g_v^{J(v, \alpha)})^{q-m}(\alpha)\dfrac{||x(\alpha)||_{v}^{md}}{\prod_{j=1}^{\kappa}||Q_{i_j(v, \alpha)}(\alpha)(x(\alpha))||_v\prod_{j=\kappa+1}^{m-n+\kappa}||Q_{i_j(v, \alpha)}(\alpha)(x(\alpha))||_v}\\\nonumber
&&\times \dfrac{1}{\prod_{j=m-n+\kappa+1}^{m}||Q_{i_j(v, \alpha)}(\alpha)(x(\alpha))||_v}\notag\\\nonumber
&\le& (g_v^{J(v, \alpha)})^{q-m}(\alpha)h_v^{n-\kappa-1}(\alpha)\dfrac{||x(\alpha)||_{v}^{md}}{
\prod_{j=1}^{n}||P_{J(v, \alpha), j}(\alpha)(x(\alpha))||_v\prod_{j=\kappa+1}^{m-n+\kappa}||Q_{i_j(v, \alpha)}(\alpha)(x(\alpha))||_v}.
\end{eqnarray}

From (\ref{ctm3}), we get
\begin{eqnarray}\label{a1}
\sum_{i=1}^{q}\lambda_{D_i(\alpha),v}({\bf x}(\alpha))\le \sum_{j=1}^{n}\lambda_{D_{J(v,\alpha)},j,v}({\bf x}(\alpha))+\sum_{j=1}^{m-n}\lambda_{D_{J(v,\alpha)},j,v}({\bf x}(\alpha))+o(h({\bf x}(\alpha)))
\end{eqnarray}
if $m-n\le \kappa$ and
\begin{eqnarray}\label{a2}
&&\sum_{i=1}^{q}\lambda_{D_i(\alpha),v}({\bf x}(\alpha))\notag\\&\le& \sum_{j=1}^{n}\lambda_{D_{J(v,\alpha)},j,v}({\bf x}(\alpha))+\sum_{j=\kappa+1}^{m-n+\kappa}\lambda_{D_{i_j(v,\alpha),v}(\alpha)}({\bf x}(\alpha))\notag\\
&\le& \sum_{j=1}^{n}\lambda_{D_{J(v,\alpha)},j,v}({\bf x}(\alpha))+\sum_{j=\kappa+1}^{m-n+\kappa}\lambda_{D_{i_j(v,\alpha),v}(\alpha)}({\bf x}(\alpha))
+o(h({\bf x}(\alpha)))\notag\\
&\le& \sum_{j=1}^{n}\lambda_{D_{J(v,\alpha)},j,v}({\bf x}(\alpha))+\dfrac{m-n}{\kappa}\sum_{j=\kappa+1}^{2\kappa}\lambda_{D_{i_j(v,\alpha),v}(\alpha)}({\bf x}(\alpha))
+o(h({\bf x}(\alpha)))\notag\\
&=&\sum_{j=1}^{n}\lambda_{D_{J(v,\alpha)},j,v}({\bf x}(\alpha))+\dfrac{m-n}{\kappa}\sum_{j=1}^{\kappa}\lambda_{D_{J(v,\alpha)},j,v}({\bf x}(\alpha))
+o(h({\bf x}(\alpha)))
\end{eqnarray}
if $m-n\ge\kappa.$ Combine (\ref{a1}) and (\ref{a2}), we get
\begin{eqnarray}\label{a3}
\sum_{i=1}^{q}\lambda_{D_i(\alpha),v}({\bf x}(\alpha))&\le& \sum_{j=1}^{n}\lambda_{D_{J(v,\alpha)},j,v}({\bf x}(\alpha))+o(h({\bf x}(\alpha)))\notag\\&&+\dfrac{m-n}{\max\{1,\min \{m-n,\kappa\}\}}\sum_{j=1}^{n}\lambda_{D_{J(v,\alpha)},j,v}({\bf x}(\alpha)).
\end{eqnarray}
From (\ref{a3}), we obtain
\begin{eqnarray*}
&\log \prod_{j=1}^{q}\dfrac{||x(\alpha)||_v^{d}}{||Q_j(\alpha)(x(\alpha))||_v}\\
&\le \left(\dfrac{m-n}{\max\{1,\min \{m-n,\kappa\}\}}+1\right)\log \prod_{j=1}^{n}\dfrac{||x(\alpha)||_v^{d}}{||P_{J(v, \alpha), j}(\alpha)(x(\alpha))||_v}+o(h({\bf x}(\alpha))).
\end{eqnarray*}
This implies
\begin{eqnarray}\label{ct22s}
&&\log \prod_{j=1}^{q}||Q_j(\alpha)(x(\alpha))||_v\\&\ge& \left(q-\left(\dfrac{m-n}{\max\{1,\min \{m-n,\kappa\}\}}+1\right)n\right)d\log ||x(\alpha)||_v\notag\\
&&+(\dfrac{m-n}{\max\{1,\min \{m-n,\kappa\}\}}+1)\log \prod_{j=1}^{n}||P_{J(v, \alpha), j}(\alpha)(x(\alpha))||_v+o(h({\bf x}(\alpha)))\notag.
\end{eqnarray}

By Lemma \ref{L4s}, there exist homogeneous polynomials $\phi_1^{J(v, \alpha)}, \dots, \phi_{H_V(N)}^{J(v, \alpha)}$(depend on $J(v, \alpha)$) in $\mathcal R_{A, \Cal Q}[x_0, \dots, x_M]_N$ and functions $u(N), v(N)$ (common for all $J(v,\alpha))$ such that $\{\phi_i^{J(v, \alpha)}\}$  a basis of $\mathcal R_{A, \Cal Q}-$ vector space $\dfrac{\mathcal R_{A, \Cal Q}[x_0, \dots, x_M]_N}{\mathcal I_{A, \Cal Q}(V)_N}$ and
$$ \prod_{\ell=1}^{H_V(N)}\phi_\ell^{J(v, \alpha)}-(P_{J(v,\alpha), 1}\dots P_{J(v, \alpha), n})^{\dfrac{\deg V \cdot N^{n+1}}{d(n+1)!}-u(N)}P_{J(v,\alpha)} \in  \mathcal I_{A, \Cal Q}(V)_N,$$
where  $P_{J(v,\alpha)}\in \mathcal R_{A, \Cal Q}[x_0,\dots,x_M]$ is a homogeneous polynomials of degree
$\frac{\deg V\cdot N^{n+1}}{(n+1)!}+v(N), v(N)=O(N^n).$ Thus, for all ${\bf x}(\alpha) \in V(k)$, we have
\begin{eqnarray*} \prod_{\ell=1}^{H_V(N)}\phi_\ell^{J(v, \alpha)}(\alpha)(x(\alpha))=\Big(\prod_{i=1}^nP_{J(v,\alpha), i}(\alpha)(x(\alpha))\Big)^{\dfrac{\deg V \cdot N^{n+1}}{d(n+1)!}-u(N)}P_{J(v,\alpha)}(\alpha)(x(\alpha)) .
\end{eqnarray*}\par

On the other hand, it is easy to see that there exist $h_{J(v, \alpha)}\in \mathcal C_{\bf x}$ such that
\begin{eqnarray*}
&&||P_{J(v,\alpha)}(\alpha)(x(\alpha))||_v \\&\leq& ||x(\alpha)||_v^{\deg P_{J(v,\alpha)}}h_{J(v, \alpha)}(\alpha)\\
&=&||x(\alpha)||_v^{\frac{\deg V\cdot N^{n+1}}{(n+1)!}+v(N)}h_{J(v, \alpha)}(\alpha).
\end{eqnarray*}
Therefore,
\begin{eqnarray*}
&&\log \prod_{\ell=1}^{H_V(N)}||\phi_\ell^{J(v, \alpha)}(\alpha)(x(\alpha))||_v\\
&\leq&\Big( \dfrac{\deg V . N^{n+1}}{d(n+1)!}-u(N)\Big)\cdot\log ||\prod_{i=1}^nP_{J(v,\alpha), i}(\alpha)(x(\alpha))||_v\\
&&+\log^+ h_{J(v, \alpha)}(\alpha)+( \dfrac{\deg V . N^{n+1}}{(n+1)!}+v(N))\log ||x(\alpha)||_v
\end{eqnarray*}
This implies that there are functions $\omega_1(N),\omega_2(N)\leq O(\frac{1}{N})$ such that
\begin{eqnarray}\label{ct23s}
&&\log ||\prod_{i=1}^nP_{J(v,\alpha), i}(\alpha)(x(\alpha))||_v\notag\\&\geq&  \Big(\dfrac{d(n+1)!}{\deg V . N^{n+1}}+\omega_1(N) \Big)\cdot\log \prod_{\ell=1}^{H_V(N)}||\phi_\ell^{J(v, \alpha)}(\alpha)(x(\alpha))||_v \notag\\
&&-\Big(\dfrac{d(n+1)!}{\deg V . N^{n+1}}+\omega_1(N) \Big)\log ^+ h_{J(v, \alpha)}(\alpha)-(d+\omega_2(N))\log ||x(\alpha)||_v.
\end{eqnarray}

\noindent By (\ref{ct22s}) and (\ref{ct23s}), we have
\begin{eqnarray}\label{ct26s}
&&\log \prod_{j=1}^{q}||Q_j(\alpha)(x(\alpha))||_v\\&\ge& \left(q-\left(\dfrac{m-n}{\max\{1,\min \{m-n,\kappa\}\}}+1\right)(n+1)\right)d\log ||x(\alpha)||_v\notag\\
&&+\left(\dfrac{m-n}{\max\{1,\min \{m-n,\kappa\}\}}+1\right)\Big(\Big(\dfrac{d(n+1)!}{\deg V . N^{n+1}}+\omega_1(N)\Big)\notag\\&&\times\log \prod_{\ell=1}^{H_V(N)}||\phi_\ell^{J(v, \alpha)}(\alpha)(x(\alpha))||_v \notag\\
&&-\Big(\dfrac{d(n+1)!}{\deg V . N^{n+1}}+\omega_1(N) \Big)\log^+ h_{J(v, \alpha)}(\alpha)-\omega_2(N)\log ||x(\alpha)||_v\Big)+o(h({\bf x}(\alpha)))\notag.
\end{eqnarray}

We fix homogeneous polynomials $\Phi_1, \dots, \Phi_{H_V(N)} \in \mathcal R_{A, \{Q_j\}_{j=1}^{q}}[x_0, \dots, x_M]_N$ such that they form a basis of $\mathcal R_{A, \Cal Q}-$ vector space $\dfrac{\mathcal R_{A, \Cal Q}[x_0, \dots, x_M]_N}{{{\mathcal I}_{A, \Cal Q}(V)}_N}.$ Then, there exist homogeneous linear polynomials
$$L_1^{J(v, \alpha)}, \dots, L_{H_V(N)}^{J(v, \alpha)}\in \mathcal R_{A, \Cal Q}[y_1, \dots, y_{H_V(N)}]$$ such that they are linear independent over $\mathcal R_{A, \{Q_j\}_{j=1}^{q}}$ and
$$ \phi_\ell^{J(v, \alpha)}-L_\ell^{J(v,\alpha)}(\Phi_1, \dots, \Phi_{H_V(N)}) \in \mathcal I_{A, \Cal Q}(V)_N,$$
for all $\ell=1, \dots, H_V(N).$ It is clear that  $h(L_\ell^{J(v, \alpha)}(\beta))=o(h(x(\beta))), \beta\in A$ and $A$ is coherent with respect to $\{L_\ell\}_{\ell=1}^{H_V(N)}.$

We have,
\begin{eqnarray}\label{ct24s}
 \prod_{\ell=1}^{H_V(N)}||\phi_\ell^{J(v, \alpha)}(\alpha)(x(\alpha))||_v=\prod_{\ell=1}^{H_V(N)}||L_\ell^{J(v, \alpha)}(\Phi_1, \dots, \Phi_{H_V(N)})(\alpha)(x(\alpha))||_v.
\end{eqnarray}
We write
$$L_\ell^{J(v, \alpha)}(y_1, \dots, y_{H_V(N)})=\sum_{s=1}^{H_V(N)}g_{\ell r}y_s, \quad g_{\ell s}\in\mathcal R_{A, \Cal Q}.$$
Since $L_1^{J(v, \alpha)}, \dots, L_{H_N(V)}^{J(v, \alpha)}$ are linear independent over $\mathcal R_{A, \Cal Q},$   we have $\det (g_{\ell s})\ne 0\in \mathcal R_{A,\Cal Q}.$ Thus, due to coherent property of $A,$ $\det(g_{\ell s})(\beta)\ne 0$  for all $\beta\in A$, outside a finite subset of $A.$ By passing to an infinite subset if necessary, we may assume that $L_1^{J(v,\alpha)}(\beta),\dots,L_{H_V(N)}^{J(v,\alpha)}(\beta)$ are lineraly independent over $k$ for all $\beta\in A.$

Now we consider the sequence of points $F(\alpha)=[\Phi_1(x(\alpha)), \dots, \Phi_{H_N(V)}(x(\alpha))]$ from $A$ to $\P^{H_V(N)-1}(k)$ and
moving hyperplanes $\Cal L:=\{L_1^{J(v, \alpha)}, \dots, L_{H_N(V)}^{J(v, \alpha)}\}$ in $\P^{H_V(N)-1}(k),$ indexed by $A.$ We claim that $F$ is linearly nondegenerate with respect to $\Cal L.$ Indeed, ortherwise, then there is a linear form $L\in\Cal R_{B, \Cal L}[y_1,\dots, y_{H_N(V)}]$ for some infinite coherent subset $B\subset A,$ such that $L(F)|_B\equiv 0$ in $B,$ which contradicts to the assumption that $x$ is algebraically nondegenerate with respect to $\Cal Q.$

By Theorem A, for any $\epsilon >0,$ there is an infinite subset of $A$ (common for all $J(v,\alpha)),$ denoted again by $A$, such that
\begin{eqnarray}\label{ct25s}
\sum_{v\in S} \log \prod_{\ell=1}^{H_V(N)}\dfrac{||F(\alpha)||_v||L_\ell^{J(v, \alpha)}(\alpha)||_v}{||L_\ell^{J(v, \alpha)}(\alpha)(F(\alpha))||_v}\le (H_V(N)+\epsilon)h(F(\alpha)),
\end{eqnarray}
for all $\alpha \in A.$\par

Combining with (\ref{ct26s}) and (\ref{ct25s}) we have
\begin{eqnarray}\label{ct27s}
&&\sum_{v\in S}\sum_{j=1}^{q}\log \dfrac{||x(\alpha)||_v^{d}}{||Q_j(\alpha)(x(\alpha))||_v}\notag\\
&\le&\left(\dfrac{m-n}{\max\{1,\min \{m-n,\kappa\}\}}+1\right)(n+1)d \sum_{v\in S}\log ||x(\alpha)||_v\notag+\omega_2(N)\sum_{v\in S}\log ||x(\alpha))||_v\notag\\
&&+\left(\dfrac{m-n}{\max\{1,\min \{m-n,\kappa\}\}}+1\right)(\dfrac{d(n+1)!}{\deg V \cdot N^{n+1}}+\omega_1(N))\notag\\
&&\times\sum_{v\in S}\log \prod_{\ell=1}^{H_V(N)}\dfrac{||F(\alpha)||_v||L_\ell^{J(v,\alpha)}(\alpha)||_v}{||L_\ell^{J(v, \alpha)}(\alpha)(x(\alpha))||_v}\notag\\
&&-H_V(N)\left(\dfrac{m-n}{\max\{1,\min \{m-n,\kappa\}\}}+1\right)\left(\dfrac{d(n+1)!}{\deg V \cdot N^{n+1}}+\omega_1(N)\right)\sum_{v\in S}\log ||F(\alpha)||_v\notag\\
&&+o(h(x(\alpha))).
\end{eqnarray}
Since the above inequality is independent of the choice of components of ${\bf x}(\alpha)$, we can choose the components of ${\bf x}(\alpha)$ being $S$-integers so that
\begin{eqnarray}\label{ctm4}
\sum_{v\in S}\log ||x(\alpha)||_v&=&h({\bf x}(\alpha))+O(1),  \text{and}\notag\\
\sum_{v\in S}\log\Vert F(\alpha)\Vert_v&\le& h(F(\alpha))+O(1) \le Nh({\bf x}(\alpha))+O(1).
\end{eqnarray}

Combining with (\ref{ct27s}) and (\ref{ctm4}), we obtain
\begin{eqnarray*}
&&\sum_{v\in S}\sum_{j=1}^{q}\log \dfrac{||x(\alpha)||_v^{d}}{||Q_j(\alpha)(x(\alpha))||_v}\\&\le& \left(\dfrac{m-n}{\max\{1,\min \{m-n,\kappa\}\}}+1\right)(n+1)dh({\bf x}(\alpha))+\omega_2(N) h({\bf x}(\alpha))\\
&&+\left(\dfrac{m-n}{\max\{1,\min \{m-n,\kappa\}\}}+1\right)\Big(\dfrac{d(n+1)!}{\deg V \cdot N^{n+1}}+\omega_1(N)\Big)(H_V(N)+\epsilon)h(F(\alpha))\\
&&-H_V(N)\left(\dfrac{m-n}{\max\{1,\min \{m-n,\kappa\}\}}+1\right)\Big(\dfrac{d(n+1)!}{\deg V \cdot N^{n+1}}+\omega_1(N)\Big)h(F(\alpha))\\
&&+o(h({\bf x}(\alpha)).
\end{eqnarray*}
Hence, we obtain
\begin{eqnarray}\label{ct28s}
&&\sum_{v\in S}\sum_{j=1}^{q}\log \dfrac{||x(\alpha)||_v^{d}||Q_j(\alpha)||_v}{||Q_j(\alpha)(x(\alpha))||_v}\\&\le&  \left(\dfrac{m-n}{\max\{1,\min \{m-n,\kappa\}\}}+1\right)(n+1)dh({\bf x}(\alpha))+\omega_2(N)h({\bf x}(\alpha))\notag\\
&&+\epsilon\left(\dfrac{m-n}{\max\{1,\min \{m-n,\kappa\}\}}+1\right)\Big(\dfrac{d(n+1)!}{\deg V . N^{n+1}}+\omega_1(N)\Big)h(F(\alpha))+o(h({\bf x}(\alpha)))\notag.
\end{eqnarray}
Here, we note that $h(D_j(\alpha))=o(h({\bf x}(\alpha)))$ and there is only finite points ${\bf x}$ such that $h({\bf x})$ is bounded, combining (\ref{ctm4}) and (\ref{ct28s}), for any $\varepsilon>0,$ by our choice with $\omega_1,\omega_2$ for  $N$ large enough, we get
\begin{eqnarray*}
\sum_{v\in S}\sum_{j=1}^{q}\log \dfrac{||x(\alpha)||_v^{d}||Q_j(\alpha)||_v}{||Q_j(\alpha)(x(\alpha))||_v}\le  \left(\left(\dfrac{m-n}{\max\{1,\min \{m-n,\kappa\}\}}+1\right)(n+1)+\varepsilon\right)d h({\bf x}(\alpha)),
\end{eqnarray*}
for all $\alpha\in A$ outside finite set. This completes the proof of Theorem \ref{Schmidt}.
\end{proof}

\noindent Tingbin Cao\\
Department of Mathematics\\
 Nanchang University\\
Jiangxi 330031, P. R. China\\
e-mail: tbcao@ncu.edu.cn\\

\noindent Nguyen Van Thin\\
Department of Mathematics\\
 Thai Nguyen University of Education\\
Luong Ngoc Quyen Street,
Thai Nguyen city, Vietnam.\\
e-mail: thinmath@gmail.com\\


\begin{thebibliography}{99}\bibitem{CRY2} Z. Chen, M. Ru, and Q. Yan, \textit{The degenerated second main theorem and Schmidt's subspace theorem}, Sci. China Math. \textbf{55} (2012), 1367-1380.


\bibitem{CRY3} Z. Chen, M. Ru, and Q. Yan, \textit{Schmidt's subspace theorem with moving targets},  Internat. Math. Res. Notices \textbf{15}(2015), no. 1, 6305-6329, 2015.

\bibitem{DT1}G. Dethloff, and T. V. Tan, \textit{Holomorphic curves into algebraic varieties intersecting moving hypersurface targets}, To appear in Acta Math. Vietnamica (2019), Doi: 10.1007/s40306-019-00336-3

 \bibitem{CZ} P. Corvaja and U. Zannier, \textit{On a general Thue's equation}, Amer. J. Math. \textbf{126} (2004),
1033-1055.

\bibitem{DT}  G. Dethloff and T. V. Tan,   \textit{A second main theorem
for moving hypersurface targets}, Houston J. Math. \textbf{37} (2011), 79-111.



\bibitem{EF} J. H. Evertse and R. G. Ferretti, \textit{Diophantine inequalities on projective varieties}, Internat. Math. Res. Notices \textbf{25}(2002), 1295-1330.

\bibitem{EF2} J. H. Evertse and R. G. Ferretti, \textit{A generalization of the subspace theorem with polynomials of higher degree},
Developments in Mathematics, \textbf{16} (2008), 175-198, Springer-Verlag, New York.





\bibitem{G} G. Le, \textit{Schmidt's subspace theorem with moving hypersurfaces}, Int. J. Number Theory \textbf{11}(2015), 139-158.









\bibitem{Q2} S. D. Quang, \textit{Schmidt's subspace theorem for moving hypersurface in subgeneral position}, Int. J. Number Theory \textbf{14}(2018), 103-121, 2018.


\bibitem{R2} M. Ru, \textit{A defect relation for holomorphic curves intersecting hypersurfaces}, Amer. J. Math. \textbf{126}(2004),
215-226.

\bibitem{R3} M. Ru, \textit{Holomorphic curves into algebraic varieties}, Ann of Math. \textbf{169}(2009), 255-267.

\bibitem{RV} M. Ru and P. Vojta, \textit{Schmidt's subspace theorem with moving targets}, Invent. Math. \textbf{127}(1997), 51-65.





\bibitem{STT} N. T. Son, T. V. Tan and N. V. Thin, Schmidt's subspace theorem for moving hypersurface targets, J. Number Theory {\bf186}(2018), 346-369.





\bibitem{Z} O. Zariski, \textit{Generalized weight properties of the resultant of $n+1$ polynomials in $n$ indeterminates}, Trans. AMS \textbf{41}(1937), 249-265.

\bibitem{XC} L. B. Xie and T. B. Cao, \textit{Second Main Theorem for holomorphic curves into algebraic varieties intersecting moving hypersurfaces targets}, Arxiv: 1908.05844v2.

\end{thebibliography}
\end{document}